\documentclass[12pt,fleqn]{scrartcl}

\usepackage[left=2cm,right=2cm,top=1.2cm,bottom=1.1cm,includeheadfoot]{geometry}
\usepackage[utf8x]{inputenc}
\usepackage[english]{babel}

\usepackage{amssymb,amsmath,amstext,amsthm,amsfonts}
\usepackage{pict2e}

\usepackage{mathtools}

\usepackage{pdfpages}
\usepackage{graphicx}

\newtheorem{thm}{Theorem}[section]
\newtheorem{cor}[thm]{Corollary}
\newtheorem{lem}[thm]{Lemma}
\newtheorem{prop}[thm]{Proposition}

\theoremstyle{definition}
\newtheorem{defi}[thm]{Definition}
\newtheorem{rem}[thm]{Remark}
\newtheorem{exam}[thm]{Example}

\newcommand{\R}{\mathbb{R}}
\newcommand{\N}{\mathbb{N}}

\usepackage{tikz}

\usepackage{tkz-graph}
\GraphInit[vstyle = Shade]
\tikzset{
  LabelStyle/.style = {rectangle, rounded corners, draw,
                        minimum width = 2em, 
                        text = black, font = \bfseries },
  VertexStyle/.style = {circle, draw, 
                                font =  \large\bfseries},
  EdgeStyle/.style = {->, bend left} }
\definecolor{rot}{rgb}{1.000,0.000,0.000}

\definecolor{gruen}{rgb}{0.000,1.000,0.000}

\newcommand{\one}{1\!\! 1}

\usepackage{fancyhdr}
\begin{document}
\title{Combinatorial considerations on \\the invariant measure of \\a stochastic matrix  \thanks{Acknowledgement:
This research has been supported by Deutsche Forschungsgemeinschaft (DFG) through grant CRC 1114 ''Scaling Cascades in Complex Systems'', Project C05 ''Effective models for materials and interfaces with multiple scales''. }}
\author{Artur Stephan\\
\vspace{-0.3cm}
\small{Weierstrass Institute for Applied Analysis and Stochastics}\\
\small{\texttt{ artur.stephan@wias-berlin.de}}}

\maketitle

\begin{abstract}
The invariant measure is a fundamental object in the theory of Markov processes. In finite dimensions a Markov process is defined by transition rates of the corresponding stochastic matrix. The Markov tree theorem provides an explicit representation of the invariant measure of a stochastic matrix. In this note, we given a simple and purely combinatorial proof of the Markov tree theorem. In the symmetric case of detailed balance, the statement and the proof simplifies even more.
\end{abstract}

~



\section{A stochastic matrix and its invariant measure}

We consider a finite state space $Z:=\{1, \dots, N\}$ where the number of species $N\in\mathbb N$ is fixed. A stochastic matrix $M=(m_{ij})_{i,j=1, \dots N}$ (also called Markov operator) on $\R^N$ is a real matrix with non-negative entries and which satisfies $M \one =\one$, where $\one := (1,\dots, 1)^T$. This condition is equivalent to the fact that its adjoint $M^*$ maps the set of probability vectors, i.e. non-negative vectors $v\in\R^N$ with $\sum_{j=1}^N v_j=1$, to itself. See \cite{KemenySnell, Norris} for introductive reading on Markov Chains.

It is well-known that there is always a probability vector $w$ such that $M^*w=w$ or equivalently $w^T M = w^T$. The famous Theorem of Frobenius-Perron states that the eigenvector is positive if $A$ is irreducible. 
\begin{thm}[Perron (1907) - \cite{Perron}, Frobenius (1912) - \cite{Frobenius}]
Let $A\in\R^{N\times N}\geq 0$ be an irreducible matrix with spectral radius $\rho(A)$. Then $\rho(A)$ is a simple eigenvalue of the matrix $A$, the corresponding eigenspace is one-dimensional and there is a positive eigenvector.
\end{thm}

The normalized vector $w$ satisfying $w^T M = w^T$ is called the invariant measure of the stochastic matrix. The invariant measure is of great importance for stochastic processes. For a given Markov operator $M$ and initial state $p_0$, the sequence $p_n=M^{*n}p_0$ is called Markov chain and $p_\infty :=\lim_{n\rightarrow \infty} p_n$ is an invariant measure  of $M^*$. Moreover, invariant measures are also stationary measures, i.e. for $p_0 = w$ the chain is constant.

A Markov process (sometimes also called continuous time Markov chain) is given by a family $T(t) = \mathrm{e}^{tA}$ of Markov operators. The Theorem of Kakutani-Markov provides the existence of an invariant measure $w$ such that $w^TT(t)=w^T$ for any $t\geq 0$. This is equivalent to $A^*w=0$ or $w^TA=0$, where $A= T'(0)$ is the generator of the semigroup, a Markov generator. Hence, it is an element of the null space of the generator of $A$. If $M$ is a stochastic matrix (or Markov operator) then $A=M-I$ is a Markov generator. Conversely, if $A$ is a bounded Markov generator then there is a positive number $\alpha>0$ such that $M=\alpha A + I$ is a stochastic matrix. That means both of these problems, finding the null space of a Markov generator and finding the invariant measure of a Markov operator can be solved equally. But note the set of stochastic matrices is larger than the set of operators represented by $\mathrm{e}^{tA}$ with some $t$ and a Markov generator $A$. For example, there is no $t\geq0$ and Markov generator $A$ with $\mathrm{e}^{tA}=\begin{pmatrix} 0&1\\1&0 \end{pmatrix}$.

The Theorem of Frobenius-Perron is a pure existence result. An explicit formula for $w\in\R^N$
is provided by the so-called \textit{Markov tree theorem} (see Section \ref{SectionMarkovTreeTheorem} for the exact statement). The Markov Tree Theorem has the Theorem of Frobenius and Perron as an immediate consequence. 

There are many different proofs for the Markov tree theorem: algebraic proofs (like in \cite{ KruckmanGreenwaldWicks}) which compute determinants and minors and are similar to Kichhoff's proof of the the Kirchhoff's Matrix Tree Theorem \cite{Kirchhoff} (see e.g. \cite{Bollobas} for a smooth version of Kirchhoff's Matrix Tree Theorem), and stochastic proofs \cite{AnantharamTsoucas} which define a Markov process on the set of trees and investigate its time reversal. The aim of this note is to give an easy proof which is purely combinatorial. Moreover, a similar reasoning can be used for determining the invariant measure of a symmetric stochastic process, i.e. where the corresponding stochastic matrix is detailed balanced (see Section \ref{SectionDetailedBalance}).

\section{A stochastic matrix and the corresponding reaction graph}\label{SectionStochMatrixAndReactionGraph}

The entries $m_{ij}$ of a stochastic matrix correspond to transition probabilities from the state $i$ to $j$. It is convenient to illustrate the action of a stochastic matrix with a reaction network or graph. So let us recall some graph theory (see e.g. \cite{Bollobas} for further references). A graph $\gamma=(V,E)$ consists of vertices $v\in V(\gamma)$ and edges $e\in E(\gamma)$. We have finitely many vertices that are labelled with $i\in Z=\{1, \dots, N\}$. The edge $e$ going from $i$ to $j$ is often just denoted by $e_{ij}$. The transition probability $m_{ij}$ correspond to the edge $e_{ij}$. If $m_{ij}=0$, there is no edge in the graph. Clearly, we deal with directed graphs, i.e. with graphs where edges $e_{ij}$ and $e_{ji}$ can be distinguished. In Section \ref{SectionDetailedBalance} we deal also with undirected graphs where the edges $e_{ij}$ and $e_{ji}$ are not distinguished.

A (directed) path in a graph $\gamma$ is a subset of vertices $i_1, \dots, i_m$ such that $e_{i_1i_2}, \dots, e_{i_{m-1}i_m}\in E(\gamma)$. 
Two states $i$ and $j$ communicate if there is a directed path from $i$ to $j$ and a directed path from $j$ to $i$. Clearly, this defines an equivalent relation on the state space $Z$ and hence, the state space $Z$ decomposes into disjoint equivalent classes $C_1, \dots, C_m$ of states which communicate.

It can happen that some of the classes are totally disconnected to other classes, i.e. there is no path in any direction. It can also happen that some classes are connected in the sense that there is a connection only in one direction, i.e. a connection from one class to another but certainly not back. This is sometimes called \textit{weakly connected}. Let $Z_1$ be the union of all classes $C_k$ that are totally disconnected; $Z_2$ the union of all classes $C_k$ such that there are only paths ending in $C_k$ and not starting out of $C_k$; $Z_R$ the union of all remaining classes. So $Z=Z_1\cup Z_2 \cup Z_R$ with (maybe after renumbering) $Z_1 = C_1\cup\dots \cup C_k$, $Z_2 = C_{k+1}\cup\dots \cup C_{k+l}$ and $Z_R = C_{k+l+1}\cup\dots \cup C_m$. 

With this definition we get  the following general form of (the adjoint of) a stochastic matrix
\begin{align*}
M^* = 
\small{\left(
\begin{array}{ccc|ccc|cccc}
\boxed{M_1^*} & 0 & 0 & 0 & \dots & 0 & 0 & 0 & 0 & 0\\
0 & \ddots & 0 & 0 & \dots & 0 & 0 & 0 & 0 & 0\\
0 & 0 & \boxed{M_k^*} & 0 & \dots & 0 & 0 & 0 & 0 & 0\\
\hline
0 & 0 & 0 & \boxed{M_{k+1}^*} &  \dots & 0 & \boxed{X} & \boxed{X} & \boxed{X} & \boxed{X}\\
0 & 0 & 0 & 0 &  \ddots & 0 & \boxed{X} & \boxed{X} & \boxed{X} & \boxed{X}\\
0 & 0 & 0 & 0 &  \dots & \boxed{M_{k+l}^*}& \boxed{X} & \boxed{X} & \boxed{X} & \boxed{X}\\
\hline
0 & 0 & 0 & 0 & 0 & 0 & \boxed{M_{k+l+1}^*} & \boxed{X} &  \dots & \boxed{X} \\
0 & 0 & 0 & 0 & 0 & 0 & 0 & \boxed{M_{k+l+2}^*} &  \dots & \boxed{X} \\
0 & 0 & 0 & 0 & 0 & 0 & 0 & 0 &  \ddots  & \boxed{X}  \\
0 & 0 & 0 & 0 & 0 & 0 & 0 & 0 &  \dots & \boxed{M_{m}^*}
\end{array}
\right)}
\end{align*}
Here the boxed entries stand for matrices and $\boxed X$ stand for (maybe different) non-zero matrices which describe the transitions between communicating classes.

The matrices $M_j$ for $j=1,\dots, k+l$ are stochastic matrices now acting on the equivalent class $C_j$. Each of them has an invariant measure by the Theorem of Frobenius-Perron. By definition, in $C_j$ all states are communicating and the stochastic matrix $M_j$ is irreducible. Hence, it has a unique invariant measure $\mu_j$ which is positive, i.e. $M^*_j \mu_j = \mu_j$. By $\widetilde\mu_j$ we denote the trivial continuation of $\mu_j$ in $\R^N$ with zeros. Obviously, it is also an invariant measure of $M$. 

\begin{prop}
 Any invariant measure of $M$ is given by a convex combination of $\widetilde \mu_j$
 \begin{align*}
  \widetilde \mu = \sum_{j=1}^{l+k}\lambda_j \widetilde \mu_j, ~~~\lambda_j\geq 0,~~\sum_{j=1}^{l+k}\lambda_j =1.
 \end{align*}
 In particular, the entries of $\widetilde  \mu$ with index larger than $k+l$ are zero. 
\end{prop}
\begin{proof}
 As above, let $M$ be given in $m\times m$ blocks.  Since $M^*_j\mu_j = \mu_j$, it also holds $M^*\widetilde \mu = \widetilde \mu$. Hence $\widetilde \mu$ defined by the above representation is indeed an invariant measure of $M$. Now, let $\eta = (\eta_1, \dots, \eta_m)^T$ be an arbitrary invariant measure of $M$.
 By the above considerations, the first $k+l$ components of $\eta$ are uniquely determined by the irreducible components $M_j^*$ by the Theorem of Frobenius-Perron. Hence, it suffices that the entries with index larger than $k+l$ are zero.
 Let us look at $\eta_m$ and assume that $\eta_m\neq 0$. We write $M^*=\begin{pmatrix} \widetilde M^*_1 & X \\ 0 &  M^*_m                       \end{pmatrix}$, where $\tilde\eta_1$ is an invariant measure of $\widetilde M_1^*$. Hence, we have
 \begin{align*}
  X \eta_m = 0, ~~ M^*_m \eta_m = \eta_m,
 \end{align*}
 where $M^*_m$ is irreducible and non-negative and $X$ is non-zero and non-negative. Since the sums of the columns in $M_m^*$ are less or equal to 1, the corresponding matrix norm is less or equal to 1. Hence, also the spectral radius of $\rho(M_m^*)$ is less or equal to 1. Since $\eta_m$ is a eigenvector to the eigenvalue 1 (it is $\neq 0$), we have $\rho(M_m^*)=1$. By the Theorem of Frobnius-Perron, we conclude that $\eta_m>0$. But this contradicts $X\eta_m =0$, since $X\neq 0$. That shows, $\eta_m =0$.
 As above, we can show iteratively that also $\eta_j=0$ for $j=k+l+1, \dots, m$. This proves the claim. 
\end{proof}

Summarizing, we showed that the invariant measure of a stochastic matrix is totally determined by the invariant measure of its irreducible components. Moreover, in each irreducible component the invariant measure is unique. The next aim is to get an explicit formula for that unique invariant measure.


\section{Rooted trees}\label{SectionRootedTrees}
In the whole section, we fix an irreducible component $C_j$, $j=1,\dots, k+l$ with the stochastic matrix $M_j$. We denote it simply with $C$, the stochastic matrix by $M=(m_{ij})$, the induced graph of the stochastic matrix is denoted by $\gamma_0$. Let us say that the number of states in $C$ is $n\in \N$.

By definition, a \textit{directed loop} is a closed directed path, i.e. a subset of edges $\{e_{i_1i_2}, \dots, e_{i_{m}i_1}\}\subset E(\gamma)$. The graph is called \textit{acyclic} if it does not contain any directed loop. In the following we consider a special subsets of acyclic graphs. We define a \textit{tree} as a connected acyclic subgraph. A special and important type of trees are the \textit{directed rooted trees}.
\begin{defi}
Fix a state $j\in C$. We define $\Gamma_j$ as the set of all directed trees rooted at $j\in C$, i.e. all graphs $\gamma$ with the following two properties:
\begin{itemize}
 \item[a)] $\gamma$ is a directed acyclic graph with $n$ vertices.
 \item[b)] Each vertex $\bar j \in C\setminus \{j\}$ has exactly one outgoing edge and $j$ has no outgoing edge.
\end{itemize}
\end{defi}
The edges in a rooted directed tree are oriented towards the root.
Note, each $\gamma\in \Gamma_j$ is a subgraph of the complete directed graph spanned by $n$ vertices, but not necessarily a subgraph of $\gamma_0$ (which is defined by the stochastic matrix $M$). A graph $\gamma\in \Gamma_j$ has $n-1$ edges in total but no edge that starts from $j$. On the other hand the graph necessarily contains an edge that ends at $j$. Otherwise, we would have a loop spanned by all other vertices what is not possible since $\gamma$ is acyclic.
\begin{exam} 
Let us fix $j=1$ and we want to look at graphs in $\Gamma_1$. There are many graphs $\gamma\in\Gamma_1$ but some of them are similar in the sense that they only differ in the permutation of states $j\leftrightarrow k$ for some $j\neq 1$ and $k\neq 1$. We call graphs that are not similar \textit{topologically different} and do not specify the vertices in the graph. The number of such different configurations is stated in the box.\\
$n=3$:\\
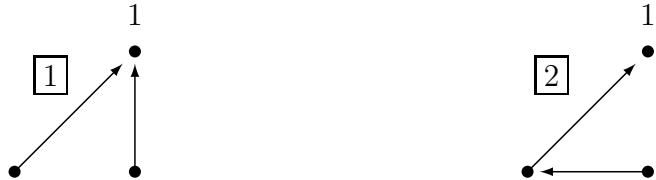
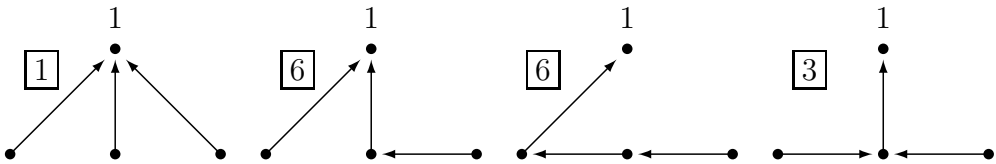
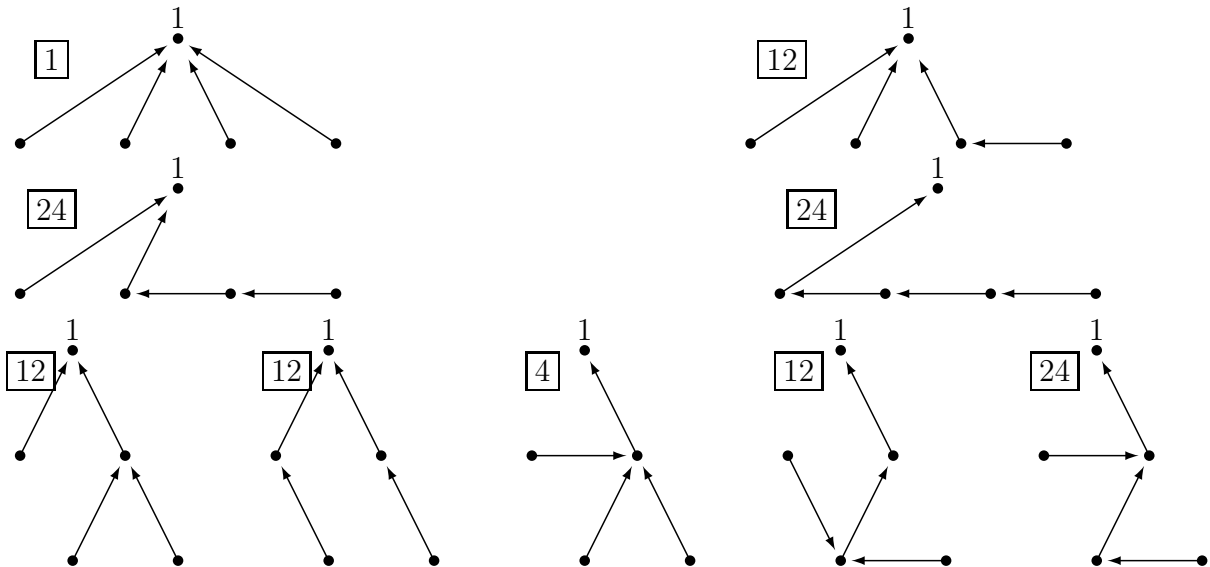
\begin{center}
\begin{minipage}[t]{12cm}
\unitlength=1.6cm
\begin{picture}(2.3, 1.1)
\linethickness{0.2mm}

\put(0.1,0.0){\circle*{0.1}} 
\put(1.1,0.0){\circle*{0.1}} 
\put(1.1,1.0){\circle*{0.1}} 

\put(0.1,0.0){\vector(1,1){0.9}} 
\put(1.1,0.0){\vector(0,1){0.9}} 

\put(1.1, 1.3){\makebox(0.0,0.0){1}}

\put(0.4,0.8){\makebox(0.0,0.0){$\boxed{1}$}}
\end{picture}\hspace{3cm}
\begin{picture}(2.3, 1.5)
\linethickness{0.2mm}

\put(0.1,0.0){\circle*{0.1}} 
\put(1.1,0.0){\circle*{0.1}} 
\put(1.1,1.0){\circle*{0.1}} 

\put(0.1,0.0){\vector(1,1){0.9}} 
\put(1.1,0.0){\vector(-1,0){0.9}}

\put(1.1, 1.3){\makebox(0.0,0.0){1}}
\put(0.3,0.8){\makebox(0.0,0.0){$\boxed{2}$}}
\end{picture}
\vspace{1cm}
\end{minipage}
\end{center}
$n=4$:\\

\begin{minipage}[t]{12cm}
\unitlength=1.4cm
\begin{picture}(2.3, 1.1)
\linethickness{0.2mm}

\put(0.1,0.0){\circle*{0.1}} 
\put(1.1,0.0){\circle*{0.1}} 
\put(2.1,0.0){\circle*{0.1}} 
\put(1.1,1.0){\circle*{0.1}} 

\put(0.1,0.0){\vector(1,1){0.9}} 
\put(1.1,0.0){\vector(0,1){0.9}} 
\put(2.1,0.0){\vector(-1,1){0.9}} 

\put(1.1, 1.3){\makebox(0.0,0.0){1}}

\put(0.4,0.8){\makebox(0.0,0.0){$\boxed{1}$}}
\end{picture}~
\begin{picture}(2.3, 1.1)
\linethickness{0.2mm}

\put(0.1,0.0){\circle*{0.1}} 
\put(1.1,0.0){\circle*{0.1}} 
\put(2.1,0.0){\circle*{0.1}} 
\put(1.1,1.0){\circle*{0.1}} 

\put(0.1,0.0){\vector(1,1){0.9}} 
\put(1.1,0.0){\vector(0,1){0.9}} 
\put(2.1,0.0){\vector(-1,0){0.9}} 

\put(1.1, 1.3){\makebox(0.0,0.0){1}}
\put(0.4,0.8){\makebox(0.0,0.0){$\boxed{6}$}}
\end{picture}~
\begin{picture}(2.3, 1.5)
\linethickness{0.2mm}

\put(0.1,0.0){\circle*{0.1}} 
\put(1.1,0.0){\circle*{0.1}} 
\put(2.1,0.0){\circle*{0.1}} 
\put(1.1,1.0){\circle*{0.1}} 

\put(0.1,0.0){\vector(1,1){0.9}} 
\put(1.1,0.0){\vector(-1,0){0.9}} 
\put(2.1,0.0){\vector(-1,0){0.9}}

\put(1.1, 1.3){\makebox(0.0,0.0){1}}
\put(0.3,0.8){\makebox(0.0,0.0){$\boxed{6}$}}
\end{picture}~
\begin{picture}(2.3, 1.1)
\linethickness{0.2mm}

\put(0.1,0.0){\circle*{0.1}} 
\put(1.1,0.0){\circle*{0.1}} 
\put(2.1,0.0){\circle*{0.1}} 
\put(1.1,1.0){\circle*{0.1}} 

\put(0.1,0.0){\vector(1,0){0.9}} 
\put(1.1,0.0){\vector(0,1){0.9}} 
\put(2.1,0.0){\vector(-1,0){0.9}} 

\put(1.1, 1.3){\makebox(0.0,0.0){1}}
\put(0.4,0.8){\makebox(0.0,0.0){$\boxed{3}$}}
\end{picture}
\vspace{1cm}
\end{minipage}
\\
$n=5$:\\
\begin{minipage}[t]{12cm}
\unitlength=1.4cm

\begin{picture}(3.3, 1.4)
\linethickness{0.2mm}
\put(0.1,0.0){\circle*{0.1}} 
\put(1.1,0.0){\circle*{0.1}} 
\put(2.1,0.0){\circle*{0.1}} 
\put(3.1,0.0){\circle*{0.1}} 
\put(1.6,1.0){\circle*{0.1}} 

\put(0.1,0.0){\vector(3,2){1.4}} 
\put(1.1,0.0){\vector(1,2){0.4}} 
\put(2.1,0.0){\vector(-1,2){0.4}} 
\put(3.1,0.0){\vector(-3,2){1.4}} 

\put(1.6, 1.2){\makebox(0.0,0.0){1}}

\put(0.4,0.8){\makebox(0.0,0.0){$\boxed{1}$}}
\end{picture}\hspace{5cm}
\begin{picture}(3.3, 1.4)
\linethickness{0.2mm}

\put(0.1,0.0){\circle*{0.1}} 
\put(1.1,0.0){\circle*{0.1}} 
\put(2.1,0.0){\circle*{0.1}} 
\put(3.1,0.0){\circle*{0.1}} 
\put(1.6,1.0){\circle*{0.1}} 

\put(0.1,0.0){\vector(3,2){1.4}} 
\put(1.1,0.0){\vector(1,2){0.4}} 
\put(2.1,0.0){\vector(-1,2){0.4}} 
\put(3.1,0.0){\vector(-1,0){0.9}} 

\put(1.6, 1.2){\makebox(0.0,0.0){1}}

\put(0.4,0.8){\makebox(0.0,0.0){$\boxed{12}$}}
\end{picture}

\begin{picture}(3.3, 1.4)
\linethickness{0.2mm}

\put(0.1,0.0){\circle*{0.1}} 
\put(1.1,0.0){\circle*{0.1}} 
\put(2.1,0.0){\circle*{0.1}} 
\put(3.1,0.0){\circle*{0.1}} 
\put(1.6,1.0){\circle*{0.1}} 

\put(0.1,0.0){\vector(3,2){1.4}} 
\put(1.1,0.0){\vector(1,2){0.4}} 
\put(2.1,0.0){\vector(-1,0){0.9}} 
\put(3.1,0.0){\vector(-1,0){0.9}} 

\put(1.6, 1.2){\makebox(0.0,0.0){1}}
\put(0.4,0.8){\makebox(0.0,0.0){$\boxed{24}$}}
\end{picture}
\hspace{5.3cm}
\begin{picture}(3.3, 1.4)
\linethickness{0.2mm}

\put(0.1,0.0){\circle*{0.1}} 
\put(1.1,0.0){\circle*{0.1}} 
\put(2.1,0.0){\circle*{0.1}} 
\put(3.1,0.0){\circle*{0.1}} 
\put(1.6,1.0){\circle*{0.1}} 

\put(0.1,0.0){\vector(3,2){1.4}} 
\put(1.1,0.0){\vector(-1,0){0.9}} 
\put(2.1,0.0){\vector(-1,0){0.9}} 
\put(3.1,0.0){\vector(-1,0){0.9}} 

\put(1.6, 1.2){\makebox(0.0,0.0){1}}

\put(0.4,0.8){\makebox(0.0,0.0){$\boxed{24}$}}
\end{picture}
\end{minipage}
\vspace{0.3cm}\\
\begin{minipage}[t]{12cm}
\unitlength=1.4cm
\begin{picture}(2.3, 2.3)
\linethickness{0.2mm}

\put(0.1,1.0){\circle*{0.1}} 
\put(1.1,1.0){\circle*{0.1}} 
\put(0.6,2.0){\circle*{0.1}} 
\put(0.6,0.0){\circle*{0.1}} 
\put(1.6,0.0){\circle*{0.1}} 

\put(0.1,1.0){\vector(1,2){0.45}} 
\put(1.1,1.0){\vector(-1,2){0.45}} 
\put(0.6,0.0){\vector(1,2){0.45}} 
\put(1.6,0.0){\vector(-1,2){0.45}} 

\put(0.6, 2.2){\makebox(0.0,0.0){1}}

\put(0.2,1.8){\makebox(0.0,0.0){$\boxed{12}$}}
\end{picture}~
\begin{picture}(2.3, 2.3)
\linethickness{0.2mm}

\put(0.1,1.0){\circle*{0.1}} 
\put(1.1,1.0){\circle*{0.1}} 
\put(0.6,2.0){\circle*{0.1}} 
\put(0.6,0.0){\circle*{0.1}} 
\put(1.6,0.0){\circle*{0.1}} 

\put(0.1,1.0){\vector(1,2){0.45}} 
\put(1.1,1.0){\vector(-1,2){0.45}} 
\put(0.6,0.0){\vector(-1,2){0.45}} 
\put(1.6,0.0){\vector(-1,2){0.45}} 

\put(0.6, 2.2){\makebox(0.0,0.0){1}}

\put(0.2,1.8){\makebox(0.0,0.0){$\boxed{12}$}}
\end{picture}~
\begin{picture}(2.3, 2.3)
\linethickness{0.2mm}

\put(0.1,1.0){\circle*{0.1}} 
\put(1.1,1.0){\circle*{0.1}} 
\put(0.6,2.0){\circle*{0.1}} 
\put(0.6,0.0){\circle*{0.1}} 
\put(1.6,0.0){\circle*{0.1}} 

\put(0.1,1.0){\vector(1,0){0.9}} 
\put(1.1,1.0){\vector(-1,2){0.45}} 
\put(0.6,0.0){\vector(1,2){0.45}} 
\put(1.6,0.0){\vector(-1,2){0.45}} 

\put(0.6, 2.2){\makebox(0.0,0.0){1}}

\put(0.2,1.8){\makebox(0.0,0.0){$\boxed{4}$}}
\end{picture}~
\begin{picture}(2.3, 2.3)
\linethickness{0.2mm}

\put(0.1,1.0){\circle*{0.1}} 
\put(1.1,1.0){\circle*{0.1}} 
\put(0.6,2.0){\circle*{0.1}} 
\put(0.6,0.0){\circle*{0.1}} 
\put(1.6,0.0){\circle*{0.1}} 

\put(0.1,1.0){\vector(1,-2){0.45}} 
\put(1.1,1.0){\vector(-1,2){0.45}} 
\put(0.6,0.0){\vector(1,2){0.45}} 
\put(1.6,0.0){\vector(-1,0){0.9}} 

\put(0.6, 2.2){\makebox(0.0,0.0){1}}

\put(0.2,1.8){\makebox(0.0,0.0){$\boxed{12}$}}
\end{picture}~
\begin{picture}(2.3, 2.3)
\linethickness{0.2mm}

\put(0.1,1.0){\circle*{0.1}} 
\put(1.1,1.0){\circle*{0.1}} 
\put(0.6,2.0){\circle*{0.1}} 
\put(0.6,0.0){\circle*{0.1}} 
\put(1.6,0.0){\circle*{0.1}} 

\put(0.1,1.0){\vector(1,0){0.9}} 
\put(1.1,1.0){\vector(-1,2){0.45}} 
\put(0.6,0.0){\vector(1,2){0.45}} 
\put(1.6,0.0){\vector(-1,0){0.9}} 

\put(0.6, 2.2){\makebox(0.0,0.0){1}}

\put(0.2,1.8){\makebox(0.0,0.0){$\boxed{24}$}}
\end{picture}

\vspace{1cm}
\end{minipage}

\end{exam}

The following lemma is easy but important.
\begin{lem}\label{LemmaDirectedPaths}
 Let $\gamma \in \Gamma_j$ for some $j\in C$ be a fixed directed rooted tree. For every vertex $\bar j\neq j$ there is a directed path in $\gamma$ which starts from  $\bar j$ and ends at $j$.
\end{lem}
\begin{proof}
Let us assume w.l.o.g that $\bar j =1$ and $j\neq1$. We prove the claim by contradiction, i.e. let us assume that there is no directed path from $\bar j=1$ to $j$ in $\gamma\in\Gamma_j$. We are going to construct a path to $j$ iteratively using the fact that every vertex in $\gamma$ except $j$ has exactly one outgoing edge by definition.  So, there is an edge starting from 1 which does not go to $j$ (by assumption) but goes to another vertex say 2. Then there is an edge starting from 2 does not go $j$ (again by assumption since otherwise there would be a path from $\bar j$ to $j$) and not to $1$ since the graph is acyclic. Hence it goes to another vertex say 3. Then there is an edge starting from 3 does not go to $j$, to 1 and to 2, hence it goes to say 4. Skipping the vertex $j$, we conclude till the last vertex $n$. But then, there is no suitable edge starting from $n$, since it can not go to any vertex $1,\dots, n-1$. We get a contradiction.
\end{proof}
In fact, the above proof also says that the path in the directed rooted tree unique.

\section{The Markov tree theorem}\label{SectionMarkovTreeTheorem}

Let a stochastic matrix $M=(m_{ij})_{i,j=1,\dots,n}$ be given, where $\# C = n \in \N$.
We define $w\in\R^n$ by
\begin{align}\label{StationMeasure}
 w_j = \sum_{\gamma\in\Gamma_j} \prod_{e_{ik}\in E(\gamma)} m_{ik}.
\end{align}
Note that $w$ is not normalized. The normalizing factor is $Z=\sum_{j=1}^n\sum_{\gamma\in\Gamma_j} \prod_{e_{ik}\in E(\gamma)} m_{ik}$ which contains all directed rooted trees.
\begin{exam}\label{Example3States}
 For $n=3$, we get $
  w=\begin{pmatrix}
     m_{21}m_{31} + m_{23}m_{31} + m_{21}m_{32}\\
     m_{12}m_{32} + m_{12}m_{31} + m_{13}m_{32}\\
     m_{13}m_{23} + m_{13}m_{21} + m_{12}m_{23}
    \end{pmatrix}.$
\end{exam}
We want to show that $w^T M = w^T$, or equivalently that
\begin{align*}
 \forall k\in C: w_k = \sum_{j\in C} m_{jk} w_j
\end{align*}
Since for any $k\in C$ it holds $\sum_{j\in Z} m_{kj} =1$, the above condition is equivalent to
\begin{align*}
\forall k\in C: \sum_{j\in C} m_{kj} w_k =  \sum_{j\in C}m_{jk}w_j.
\end{align*}

\begin{thm}[Markov tree theorem]\label{MainTheorem}
 It holds $w^T M =w^T$ for $w=(w_j)_{j=1,\dots,n}$ defined by $w_j = \sum_{\gamma\in\Gamma_j} \prod_{e_{ik}\in E(\gamma)} m_{ik}$.
\end{thm}

In the proof, we compute both sides and compare. To do so, we focus on $k=1$, but the other cases can be treated exactly the same way. We want to show
\begin{align}\label{ToShow}
 \sum_{j\geq 2} m_{1j} w_1 = \sum_{j\geq 2}m_{j1}w_j.
\end{align}

\begin{exam}
 Let us compute the left- and right-hand side for $n=3$. Using Example \ref{Example3States}, we have
 \begin{align*}
  L&HS = m_{12}w_1 + m_{13}w_1 =\\
  &= m_{12}m_{21}m_{31} + m_{12}m_{23}m_{31} + m_{12}m_{21}m_{32} + m_{13}m_{21}m_{31} + m_{13}m_{23}m_{31} + m_{13}m_{21}m_{32}\\
  \vspace{1cm}\\
  R&HS = m_{21} w_2 + m_{31} w_3 =\\
  &= m_{21}m_{12}m_{32} + m_{21}m_{12}m_{31} + m_{21}m_{13}m_{32} + m_{31}m_{13}m_{23} + m_{31}m_{13}m_{21} + m_{31}m_{12}m_{23}.
 \end{align*}
 Hence, (\ref{ToShow}) holds.
\end{exam}

Observe that in the formula (\ref{StationMeasure}) only edges that do not start in $j$ are taken into account. In the identity (\ref{ToShow}), the matrix entries that correspond to edges starting form $j$ are multiplied to $w_j$. That means, we have to treat graphs which emerge from rooted directed graphs  through adding one additional edge. If $e_{jk}$ is not an edge in a graph $\gamma$, let denote $\gamma\cup e_{jk}$ the graph that results from adding the edge $e_{jk}$ to $\gamma$. 

\begin{lem}\label{LemmaExactlyOneLoop}
 Let $\gamma\in\Gamma_j$ and $k\in C$ be arbitrary. Then the graph $\gamma \cup e_{jk}$ contains exactly one loop. Moreover, this loop goes through $j\in C$.
\end{lem}
\begin{proof}
 We know by Lemma \ref{LemmaDirectedPaths} that from every vertex in $\gamma$ there is a directed path back to $j$. Hence by adding one more edge $e_{jk}$, there will be definitively a loop. More than one loop is not possible, since there are no loops in $\gamma$ and there is only one edge starting from each vertex. So exactly one loop is in $\gamma\cup e_{jk}$. The loop goes obviously through $j$.
\end{proof}
We define two sets:
\begin{align*}
 S_1&:=\{\gamma\cup e_{1j}: \gamma\in \Gamma_1,~ j\in\{2,\dots, n\} \}\\
 S_2&:=\{\gamma_k\cup e_{k1}: \gamma_k\in \Gamma_k,\mathrm{~~for~~} k\in\{2,\dots, n\}\}
\end{align*}
Firstly, observe that any two elements in $S_1$ are different. The same holds for the elements of $S_2$. The key step is to show that both sets $S_1$ and $S_2$ are equal. 
\begin{prop}\label{PropositionS1equalsS2}
 It holds $S_1=S_2$.
\end{prop}
\begin{proof}
 The proof is done in two steps.\\
 1. Step: $S_1\subset S_2$. Let $\gamma\cup e_{1j}\in S_1$ for $\gamma\in \Gamma_1$ and $j\in\{2, \dots, N\}$ be arbitrary and fixed. That means, that one edge starting from 1 with arbitrary end is added to some graph $\gamma \in \Gamma_1$. By Lemma \ref{LemmaExactlyOneLoop}, there is exactly one loop in $\gamma\cup e_{1j}$. In particular there is one unique edge in the loop which ends at 1 and starts at say $\bar j$.
 
 Let us consider the graph $\gamma\cup e_{1j}$ without the edge $e_{\bar j1}$. We call it $\bar \gamma$ and want to show that $\bar\gamma \in \Gamma_{\bar j}$. Firstly, we observe that for each vertex $k$ apart from $\bar j$ there is in $\bar \gamma$ exactly one edge which starts at $k$. Hence, it remains to show that there is no loop in $\bar \gamma$. Look at $\gamma \cup e_{1j}$. It has exactly one loop, but removing one edge $e_{\bar j 1}$ the loop is destroyed. That shows that $\bar\gamma \in \Gamma_{\bar j}$ and hence we have shown that $\gamma\cup e_{1j} = \bar \gamma \cup e_{\bar j 1}$ for some $\bar\gamma \in \Gamma_{\bar j}$. This proves the first claim.\\
 
 2. Step: $S_2\subset S_1$. Let $\gamma_k\cup e_{k1}\in S_2$ for $\gamma_k\in \Gamma_k$ and $k\in\{2, \dots, N\}$ be arbitrary and fixed. That means, that one edge starting from $k$ with the end 1 is added to $\gamma_k\in\Gamma_k$.  As above by Lemma \ref{LemmaExactlyOneLoop}, it follows that there is one loop in $\gamma_k\cup e_{k1}\in S_2$ which goes through $k$ and hence also through 1. Moreover, the loops defines a unique edge that starts in 1 and goes to another vertex say $j$. Let us consider the graph $\gamma_k\cup e_{k1}$ without the edge $e_{1j}$. As above one can show that this graph is in $\Gamma_1$, i.e. there is $\gamma\in\Gamma_1$ such that $\gamma_k\cup e_{k1} = \gamma\cup e_{1j}$. This proves $S_2\subset S_1$.
\end{proof}
We are now able to prove Theorem \ref{MainTheorem}.
\begin{proof}
 As above mentioned, we prove only $(w^TM)_1 = w_1$, i.e. the identity (\ref{ToShow}).
 Using Proposition \ref{PropositionS1equalsS2}, we compute
 \begin{align*}
  \sum_{j\geq 2} m_{1j}w_1 &= \sum_{j\geq 2} m_{1j} \sum_{\gamma\in\Gamma_1}\prod_{e_{ik}\in E(\gamma)} m_{ik}=\sum_{j\geq 2}  \sum_{\gamma\in\Gamma_1}m_{1j}\prod_{e_{ik}\in E(\gamma)} m_{ik} =\\
  &=\sum_{j\geq 2}  \sum_{\gamma\in\Gamma_1}\prod_{e_{ik}\in E(\gamma \cup e_{1j})} m_{ik}= \sum_{\gamma \in S_1} \prod_{e_{ik}\in E(\gamma)} m_{ik}=\\
  &\stackrel{\mathclap{\tiny{S_1=S_2}}}{=}\;\sum_{\gamma \in S_2} \prod_{e_{ik}\in E(\gamma)} m_{ik} =\sum_{j\geq 2}\sum_{\gamma \in \Gamma_j} \prod_{e_{ik}\in E(\gamma\cup e_{j1})} m_{ik} =\\
  &= \sum_{j\geq 2} \sum_{\gamma\in\Gamma_j}m_{j1}\prod_{e_{ik}\in E(\gamma)} m_{ik}= \sum_{j\geq 2} m_{j1}\sum_{\gamma\in\Gamma_j}\prod_{e_{ik}\in E(\gamma)} m_{ik} = \sum_{j\geq 2} m_{j1}w_j.
 \end{align*}
\end{proof}

\begin{exam}
 Let us consider $n=5$ states with the following reaction graph and the associated stochastic matrix.
 
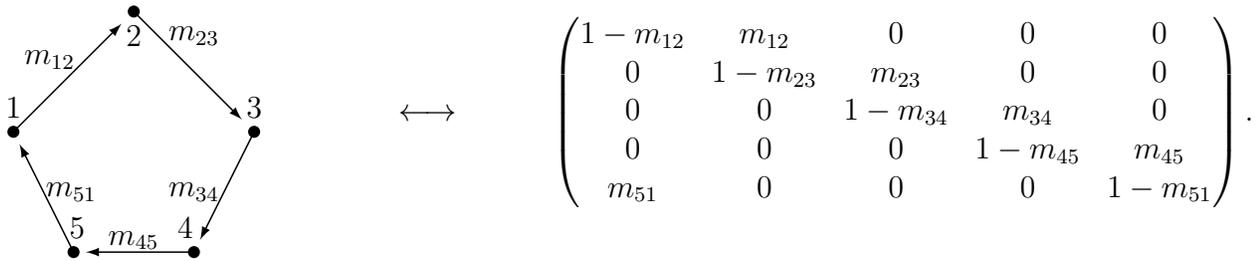
\begin{minipage}[c]{4cm}
\unitlength=1.6cm
\hspace{1cm}
\begin{picture}(2.3, 2.3)
\linethickness{0.2mm}

\put(0.1,1.0){\circle*{0.1}} 
\put(2.1,1.0){\circle*{0.1}} 
\put(1.1,2.0){\circle*{0.1}} 
\put(0.6,0.0){\circle*{0.1}} 
\put(1.6,0.0){\circle*{0.1}} 

\put(0.1,1.0){\vector(1,1){0.9}} 
\put(1.1,2.0){\vector(1,-1){0.9}} 
\put(2.1,1.0){\vector(-1,-2){0.45}} 
\put(1.6,0.0){\vector(-1,0){0.9}} 
\put(0.6,0.0){\vector(-1,2){0.45}} 

\put(0.1,1.2){\makebox(0.0,0.0){1}} 
\put(2.1,1.2){\makebox(0.0,0.0){3}} 
\put(1.1,1.8){\makebox(0.0,0.0){2}} 
\put(0.63,0.2){\makebox(0.0,0.0){5}} 
\put(1.53,0.2){\makebox(0.0,0.0){4}} 

\put(0.4,1.6){\makebox(0.0,0.0){$m_{12}$}}
\put(1.6,1.8){\makebox(0.0,0.0){$m_{23}$}}
\put(1.6,0.5){\makebox(0.0,0.0){$m_{34}$}}
\put(1.1,0.1){\makebox(0.0,0.0){$m_{45}$}}
\put(0.58,0.5){\makebox(0.0,0.0){$m_{51}$}}
\end{picture}
\end{minipage}
\hspace{1cm}
$\longleftrightarrow$
\hspace{1cm}
\begin{minipage}[c]{8cm}
$
\begin{pmatrix}
 1- m_{12}& m_{12} & 0 & 0 & 0\\
0 & 1- m_{23}& m_{23} & 0 & 0 \\
0 & 0 & 1- m_{34}& m_{34} & 0 \\
0 & 0 & 0 & 1- m_{45}& m_{45}\\
m_{51} & 0 & 0 & 0 & 1-m_{51}
\end{pmatrix}.$
\end{minipage}
\vspace{0.3cm}\newline
Formula (\ref{StationMeasure}) yields
\begin{align*}
 w&=(
 m_{23}m_{34}m_{45}m_{51},
 m_{12}m_{34}m_{45}m_{51},
 m_{12}m_{23}m_{45}m_{51},
 m_{12}m_{23}m_{34}m_{51},
 m_{12}m_{23}m_{34}m_{45})^T \\ &\sim(
 1/m_{12},
 1/m_{23},
 1/m_{34},
 1/m_{45},
 1/m_{51})
\end{align*}
as its invariant measure (despite normalization).
\end{exam}

\begin{rem}
 We want to stress that formula (\ref{StationMeasure}) always defines a vector $w$ such that $w^T M = w^T$ regardless whether $M$ is reducible or not. But it can happen that formula (\ref{StationMeasure}) defines a vector that is identically zero. This case is treated in Section \ref{SectionPositivity}.  
\end{rem}

\begin{rem}
 Proving Theorem \ref{MainTheorem}, we do not use $m_{ij}\geq 0$, i.e. also negative matrix elements are possible. The only property of the matrix $M$ that we used is that the elements in any row sum to one. Consider for example the matrix
 \begin{align*}
  M=\begin{pmatrix}
   1 & -1 & 1\\
   1 & 1 & -1\\
   -1 & 1 & 1
  \end{pmatrix}.
 \end{align*}
 We have $M=I + A$, where $A$ is an incidence matrix, often used to model electric circuits or mechanical systems of springs and masses.
 $M$ has eigenvalues $1$ and $1\pm i\sqrt 3$. Formula (\ref{StationMeasure}) yields $(1,1,1)$ as the invariant measure of $M$.
\end{rem}

\section{Positivity and Uniqueness}\label{SectionPositivity}

The aim of this section is to show whenever formula (\ref{StationMeasure}) provides a reasonable (i.e. a non-zero) vector the invariant measure is unique; or vice versa if the invariant measure is unique then formula (\ref{StationMeasure}) defines a non-zero vector. Let $\gamma_0$ be the graph defined by a given stochastic matrix $M=(m_{jk})$, i.e. $\gamma_0$ consists of all edges $e_{jk}$ such that $m_{jk}>0$. 
\begin{align*}
 Z_k = \{j\in Z: \mathrm{there ~is~a~directed~path~from~}j\mathrm{~to~}k\mathrm{~in~}\gamma_0\}.
\end{align*}
Obviously, $M$ is irreducible if and only if $\bigcap_{k\in Z} Z_k = Z$. The next proposition is helpful.
\begin{prop}\label{PropCharacterizazionPositivity}
 Let $w$ be defined as (\ref{StationMeasure}). Then
 $Z_k =Z$ if and only if $w_k >0$.
\end{prop}
\begin{proof}
We focus again on the case $k=1$.\\ 
1. Step: Let $Z_1=Z$. We want to show that $w_1> 0$.\\
The problem reduces to the following question. Let a graph $\widetilde \gamma$ with $Z_1=Z$ be given, i.e. from any vertex $j\neq 1$ there is a directed path to 1. Is it possible to obtain a subgraph which is a directed tree rooted at 1, i.e. $\gamma\in\Gamma_1$ by removing edges from $\widetilde \gamma$? It is not hard to see that this is indeed possible and actually there are many ways to construct a suitable $\gamma\in\Gamma_1$. In the following we present one possible way of construction.

For a given graph $\widetilde \gamma$ with $Z_1=Z$, let us define $V_0=\{1\}$ and let $V_1$ be the set of all vertices from which an edge to the vertex 1 starts. Collect all these edges (we call it $E_1$) and remove any other edge that starts from a vertex in $V_1$. Obviously, the graph spanned by $V_0$, $V_1$ and $E_1$ is a subgraph of a spanning tree rooted in 1. Now, let $V_2$ be the set of all vertices from which an edge to some vertex in $V_1$ starts. Collect for any vertex in $V_2$ exactly one edge that goes to some vertex in $V_1$. If there are many choices take an arbitrary one. Remove any other edge that starts from some vertex in $V_2$. We call the set of edges $E_2$. Again, it is clear that the graph spanned by $V_0$, $V_1$, $E_1$, $V_2$ and $E_2$ is a subgraph of a spanning tree rooted in 1. Now proceed as above and we get sets of vertices $V_k$ and edges $E_k$. Observe that $V_k$ contains vertices which have the distance $k$ to the vertex $1$ and hence the construction necessarily stops after at most $N-1$ steps. Define $\gamma$ as the union of $V_0$ and all $V_k$ and $E_K$. Since by assumption $Z_1=Z$ in the end any vertex is contained $\gamma$ and by construction it is clear that $\gamma\in\Gamma_1$. Hence, $w_1>0$.

2. Step: Let $w_1>0$. So, there is at least one graph $\gamma\in\Gamma_1$ such that for any edge $e_{ki}\in E(\gamma)$ it holds $m_{ki}>0$. Hence, there is a path starting from any $j\neq 1$ and ending at $1$ and we get $Z_1=Z$.
 
\end{proof}

\begin{cor}
 Let $n\geq 3$. We have $w=0$ defined by (\ref{StationMeasure}) if and only if the invariant measure is not unique.
\end{cor}
\begin{proof}
 If the invariant measure is not unique, we have at least two equivalence classes $C_1, C_2$  such that there is no path from $C_1$ to $C_2$ and no path from $C_2$ to $C_1$ (see Section \ref{SectionStochMatrixAndReactionGraph}). By Proposition (\ref{PropCharacterizazionPositivity}), we conclude $w_j = 0$ for any $j\in Z$. Hence $w=0$.
 
 Let $w=0$. If the graph is totally disconnected, then following the ideas in Section \ref{SectionStochMatrixAndReactionGraph} the invariant measure is surely not unique. So let us assume that the graph is (weakly) connected, i.e. there is at least a path in one direction connecting the different equivalence classes. We want to show that there are at least two communicating classes such that there is no path starting from them, i.e. using notation in Section (\ref{SectionStochMatrixAndReactionGraph}) they belong to $Z_+$. Surely, there is one communicating class, say $C_1$. Since $w|_{C_1} =0$, there is a state (say 2) without any path to $C_1$. Let us denote the communicating class of 2 by $C_2$ and consider the graph $\tilde \gamma$ defined by all communicating classes that can be reached from $C_2$. There is definitively a communicating class in $\tilde \gamma$ without starting paths to other communicating classes. And this communicating class is not $C_1$ since we assumed no path from $2$ to $C_1$. Hence, we found two communicating classes without starting paths, i.e. the invariant measure is not unique.
 
\end{proof}

\section{Cardinality of the sets of graphs}\label{SectionCardinality}
Now we compute the number of addends in (\ref{StationMeasure}). To do this, we introduce the following subsets of directed acyclic graphs. For $k\in\{1, \dots, n\}$ mark $k$ vertices among all $n$ vertices, say $\{j_1, \dots, j_k\}$. We define $\Gamma_{\{j_1, \dots, j_k\}}$ as a subset of directed graphs $\gamma$ with the following properties:
\begin{itemize}
 \item[a)] $\gamma$ is a directed acyclic graph with $n$ vertices.
 \item[b)] From each vertex $j \in \{j_1, \dots, j_k\}$ starts exactly one directed edge that ends at some other vertex $\bar j \in C\setminus\{j\}$.
\end{itemize}
Remembering the notation in Section \ref{SectionRootedTrees}, we see that $\Gamma_{\{1, \dots, n\}\setminus \{j\}}=\Gamma_j$.

Let us compute the cardinality of $\Gamma_{\{j_1, \dots, j_k\}}$.
\begin{prop}
 It holds $\# ~\Gamma_{\{j_1, \dots, j_k\}} = (n-k)n^{k-1}$
\end{prop}
\begin{proof}
Let $b_k^n := \#~ \Gamma_{\{j_1, \dots, j_k\}}$. 
The proof is done in two steps. Firstly, we derive a recursive formula for $b_k^n$. Secondly, we show inductively the claimed expression.
\newline
1.Step:
We define $b_0^n=1$.
 We are going to prove that $b_k^n = (n-k)n^{k-1}$.
 To compute $b_1^n$, fix one vertex, say $j_1=1$. Hence, there are $n-1$ possible edges from $j_1$. So $b_1^n=n-1$. Let us compute $b_2^n$. Fix again two vertices say 1 and 2. To define an edge from 1, we have two choices: we can go to some of the $n-2$ not marked vertices or to 2. Choosing an edge to one of the vertices that are not marked, we get for an edge from 2 then $b_1^n$ possibilities. Hence in this case $(n-2)b_1^n$. Choosing the edge to 2, we have $n-2$ options for an edge from 2. So, $b_2^n = (n-2) (b_1^n + b_0^n)$ in total. Let us compute $b_3^n$. Fix again three vertices say 1, 2 and 3. To define an edge from 1, we have two choices: we to some of the $n-3$ not marked vertices or to a marked vertex (1 or 2). 
 Choosing an edge to one not marked vertex, we get $(n-3)$ times $b_2^n$ (for the two remaining vertices) possibilities. Hence in this case $(n-2)b_2^n$. Choosing the edge to a marked vertex, we have 2 times options since we have freedom to go to 2 or to 3. Let us go to say 2. The edge from 2 can not return to 1. It can go to a 
 marked vertex, which leads to $(n-3) b_1^n$ options for an edge starting from 3. Or, it can go to a not marked vertex, which leads to $(n-3)b_0^n$ options for an edge starting from 3. Hence, $b_3^n = (n-3) (b_2^n + 2(b_1^n + b_0^n))$ in total. Stepping further, we conclude the following recursion formula for $b_k^n$: 
 \begin{align*}
  b_k^n &= (n-k)\left[b_{k-1}^n + (k-1)\left(b_{k-2}^n + (k-2) \left(b_{k-3}^n + \dots +2\left(b_1^n + b_0^n\right)\right)\dots \right)\right] = \\
  &= (n-k) \sum_{j=1}^k b_{k-j}^n \frac{(k-1)!}{(k-j)!}
 \end{align*}
2. Step: We prove inductively that $b_k^n = (n-k)n^{k-1}$. For $k=0$, we get by definition $b_0^n = 1$. We already computed $b_1^n = n-1$ and $b_2^n = (n-2)n$. Let us assume that $b_{k-j}^n = (n-k+j)n^{k-j-1}$ holds for any $j = 1, \dots, k$. We want to prove the claim for $j=0$. In particular it suffices to show that
 \begin{align*}
  \sum_{j=1}^k (n-k+j) n^{k-j-1}\frac{(k-1)!}{(k-j)!}  = n^{k-1}. 
 \end{align*}
 The left-hand side is
 \begin{align*}
   &\frac 1 n\sum_{j=1}^k (n-k+j) n^{k-j}\frac{(k-1)!}{(k-j)!} = \frac 1 n\sum_{l=0}^{k-1} (n-l) n^l\frac{(k-1)!}{l!} =\\
  =&\frac {(k-1)!}{n}\left(n + \sum_{l=1}^{k-1} \frac{n^{l+1}}{l!} - \sum_{l=1}^{k-1} \frac{n^l}{(l-1)!} \right) = \frac {(k-1)!}{n}\frac{n^k}{(k-1)!} = n^{k-1}.
 \end{align*}
This proves the claim.
\end{proof}
\begin{cor}
 It holds $\# ~\Gamma_{j} = n^{n-2}$ and hence every entry of $w$ consists of $n^{n-2}$ addends with $n-1$ factors each.
\end{cor}

\section{Symmetric case of detailed balance}\label{SectionDetailedBalance}
A special situation occurs if the Markov process satisfies a symmetry condition. We may assume that the reaction network is connected, otherwise each separated region can be treated independently. A Markov process is detailed balanced with respect to its invariant measure $w$, if by definition it is weakly reversible, i.e. whenever $m_{ij}\neq 0$ then also $m_{ji}\neq 0$, and, moreover, it holds $m_{ij}w_j = m_{ji}w_i$. This means that the stochastic matrix $M$ is symmetric in $L^2(w)$, the $L^2$ over the invariant measure $w>0$. The first property of weak reversibility implies that the invariant measure is unique. The second property, as we will see, simplifies the formula for the invariant measure hugely.

Firstly, we have the following.
\begin{lem}\label{LemOrientationDetailedBalance}
Let the stochastic matrix $M$ be detailed balanced w.r.t. the invariant measure $w>0$. Let $j_1\mapsto j_2 \mapsto \dots \mapsto j_k \mapsto j_1$ be a loop in the graph of $M$.  Then
\begin{align}
m_{j_1j_2} m_{j_2j_3} \cdots m_{j_kj_1} = m_{j_1j_k} m_{j_kj_{k-1}} \cdots m_{j_2j_1}\label{eqLoops}.
\end{align} 
\end{lem}
\begin{proof}
It holds for any $i= 1, \dots, k$ that $m_{j_ij_{i+1}}w_{j_{i+1}} = m_{j_{i+1}j_i}w_{j_{i}}$, where we use the notation that $j_{k+1}=j_1$. Taking the product of this equation for any $i= 1, \dots, k$ and dividing by $\prod_{i=1}^k w_{j_i}>0$ yields the claim. 
\end{proof}
\begin{rem}
The above relation (\ref{eqLoops}) is indeed equivalent to $M$ being detailed balanced.
\end{rem}
To simplify the formula for the invariant measure of a stochastic matrix, we need the definition of a \textit{undirected tree}. 

An  undirected graph $\gamma=(V,E)$ consists of vertices $V$ and edges $E$ where the edges $e\in E$ do not have any orientation, i.e. there is no difference between the edge going from $i$ to $j$ or from $j$ to $i$. Paths and loops can be defined as in the case of directed case before.
\begin{defi}
Let $\gamma=(V,E)$ be a undirected graph. A \textit{(undirected) tree} in $\gamma$ is a subgraph $t=(V, E')$, $E'\subset E$ containing all vertices $V$ connected by edges $e\in E'$ such that there are no loops in $t$. 
\end{defi}
Now observe the following. Pick any tree $t$ and any vertex $k$. Then the tree $t$ defines canonically a unique spanning tree rooted at $k$ by orienting each edge in $t$ into the direction of $k$. The spanning tree is easily constructed inductively by looking at the vertices on the tree which have the distance $1$, $2$ and so on to the vertex $k$. Let us denote this spanning tree by $t_k$. Now, we define $w^t = (w_k^t)_{k}$ by
\begin{align}\label{eqStationaryMeasureDetailedBalance}
w_k^t = \prod_{e_{ij}\in t_k} m_{ij}.
\end{align}
Observe that the only difference between the spanning trees defined by $i$ and $j$ is the orientation of the path between $i$ and $j$:

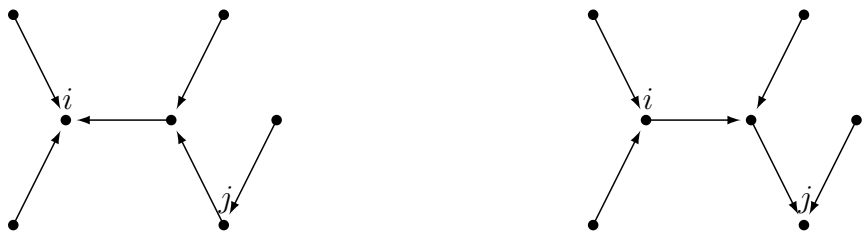
\begin{center}
\begin{minipage}[t]{12cm}
\unitlength=1.4cm





\begin{picture}(3.3, 2.3)
\linethickness{0.2mm}

\put(0.6,1.0){\circle*{0.1}} 
\put(0.1,2.0){\circle*{0.1}} 
\put(0.1,0.0){\circle*{0.1}} 
\put(1.6,1.0){\circle*{0.1}} 
\put(2.1,2.0){\circle*{0.1}} 
\put(2.1,0.0){\circle*{0.1}} 
\put(2.6,1.0){\circle*{0.1}} 

\put(0.1,0.0){\vector(1,2){0.45}} 
\put(0.1,2.0){\vector(1,-2){0.45}} 
\put(1.6,1.0){\vector(-1,0){0.9}} 
\put(2.1,2.0){\vector(-1,-2){0.45}} 
\put(2.1,0.0){\vector(-1,2){0.45}} 
\put(2.6,1.0){\vector(-1,-2){0.45}} 

\put(0.62, 1.2){\makebox(0.0,0.0){$i$}}
\put(2.12, 0.24){\makebox(0.0,0.0){$j$}}

\end{picture}\hspace{3cm}
\begin{picture}(3.3, 2.3)
\linethickness{0.2mm}

\put(0.6,1.0){\circle*{0.1}} 
\put(0.1,2.0){\circle*{0.1}} 
\put(0.1,0.0){\circle*{0.1}} 
\put(1.6,1.0){\circle*{0.1}} 
\put(2.1,2.0){\circle*{0.1}} 
\put(2.1,0.0){\circle*{0.1}} 
\put(2.6,1.0){\circle*{0.1}} 

\put(0.1,0.0){\vector(1,2){0.45}} 
\put(0.1,2.0){\vector(1,-2){0.45}} 
\put(0.6,1.0){\vector(1,0){0.9}} 

\put(2.1,2.0){\vector(-1,-2){0.45}} 
\put(1.6,1.0){\vector(1,-2){0.45}} 

\put(2.6,1.0){\vector(-1,-2){0.45}} 

\put(0.62, 1.2){\makebox(0.0,0.0){$i$}}
\put(2.12, 0.24){\makebox(0.0,0.0){$j$}}

\end{picture}
\end{minipage}
\end{center}
\vspace{2mm}

The next aim is to show that $w^t$ is indeed the (unique) invariant measure of the $M$, and moreover, that $w_k^t$ and also $w^t$ do not depend on the undirected tree $t$ chosen before, i.e. any fixed tree $t$ defines the same invariant measure up to a scaling factor.

\begin{thm}
Let the stochastic matrix $M$ be detailed balanced w.r.t. the invariant measure $w>0$. Take any tree $t$, then $w^t$, defined by (\ref{eqStationaryMeasureDetailedBalance}) is (up to normalization) the invariant measure of $M$. In fact, different trees just correspond to different normalization factors.
\end{thm}
\begin{proof} 1. Step: We show that formula (\ref{eqStationaryMeasureDetailedBalance}) defines an invariant measure.
Let us fix one tree $t$. We want to show that for any $k=1,\dots, n$ it holds
\begin{align*}
 \sum_{j\neq k} m_{kj} w_k^t = \sum_{j\neq k}m_{jk}w_j^t.
\end{align*}
To be more precise, we show that even $m_{kj}w_k^t = m_{jk}w_j^t$ holds for any $j\neq k$. Similarly to the unsymmetric case before (Lemma \ref{LemmaExactlyOneLoop}), the product $m_{kj}w_k^t$ defines a subgraph with exactly one cycle, which passes the vertices $k$ and $j$. The graph defined by $m_{jk}w_j^t$ has exactly the same structure apart from the orientation of the cycle.  But Lemma \ref{LemOrientationDetailedBalance} provides that these two products are equal, which proves the claim.

2. Step: Now, we show that formula (\ref{eqStationaryMeasureDetailedBalance}) is infact independent of the chosen tree $t$, in the sense that for any two tree $t_1$ and $t_2$ the two invariant measure are proportional. Let us take two arbitrary trees $t_1$ and $t_2$ and the associated invariant measures $w^1$ and $w^2$. Take $i\neq j$. Then the claim is equivalent to $\frac{w^1_i}{ w^1_j} = \frac{w^2_i}{ w^2_j}$.

Let us look at $w^1_i$ and $w^1_j$. Their only difference is the orientation of a path in $t_1$ connecting $i$ and $j$, i.e. $\frac{w^1_i}{w^1_j} = \frac{\mathrm{Path~in~}t^1: j\mapsto i}{\mathrm {Path~in~}t^1: i\mapsto j}$. The same relation holds for $w^2_i$ and $w^2_j$ with respect to tree $t_2$, i.e. $\frac{w^2_i}{w^2_j} = \frac{\mathrm{Path~in~}t^2: j\mapsto i}{\mathrm {Path~in~}t^2: i\mapsto j}$. So the claim is equivalent to
\begin{align*}
\frac{\mathrm{Path~in~}t^1: j\mapsto i}{\mathrm{Path~in~}t^1: i\mapsto j} =  \frac{\mathrm{Path~in~}t^2: j\mapsto i}{\mathrm {Path~in~}t^2: i\mapsto j} 
\end{align*}
or, in other words, equivalent to
\begin{align*}
(\mathrm{Path~in~}t^1: j\mapsto i )\cdot (\mathrm {Path~in~}t^2: i\mapsto j) = (\mathrm{Path~in~}t^2: j\mapsto i )\cdot (\mathrm{Path~in~}t^1: i\mapsto j).
\end{align*}
Both sides define a cycle consisting of the same edges but with different orientation. Lemma \ref{LemOrientationDetailedBalance} provides again that both terms are indeed equal.
\end{proof}
\begin{exam}
The scaling factor for different trees is not 1 in general. Consider a stochastic matrix between three states with detailed balance, i.e. $abc=def$. 

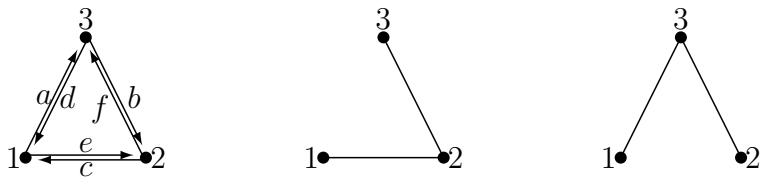
\begin{center}
\begin{minipage}[t]{12cm}
\unitlength=1.6cm

\begin{picture}(2.3, 1.5)
\linethickness{0.2mm}

\put(0.6,1.2){\circle*{0.1}} 
\put(0.1,0.2){\circle*{0.1}} 
\put(1.1,0.2){\circle*{0.1}} 

\put(0.08,0.2){\vector(1,2){0.45}} 
\put(0.62,1.2){\vector(-1,-2){0.45}} 
\put(1.08,0.2){\vector(-1,2){0.45}} 
\put(0.62,1.2){\vector(1,-2){0.45}} 
\put(1.1,0.18){\vector(-1,0){0.9}} 
\put(0.1,0.22){\vector(1,0){0.9}} 

\put(0.0, 0.2){\makebox(0.0,0.0){$1$}}
\put(1.2, 0.2){\makebox(0.0,0.0){$2$}}
\put(0.6, 1.35){\makebox(0.0,0.0){$3$}}

\put(0.25, 0.7){\makebox(0.0,0.0){$a$}}
\put(0.45, 0.7){\makebox(0.0,0.0){$d$}}
\put(0.72, 0.6){\makebox(0.0,0.0){$f$}}
\put(1.0, 0.7){\makebox(0.0,0.0){$b$}}

\put(0.6, 0.1){\makebox(0.0,0.0){$c$}}
\put(0.6, 0.3){\makebox(0.0,0.0){$e$}}

\end{picture}~
\begin{picture}(2.3, 1.3)
\linethickness{0.2mm}

\put(0.6,1.2){\circle*{0.1}} 
\put(0.1,0.2){\circle*{0.1}} 
\put(1.1,0.2){\circle*{0.1}} 

\put(0.6,1.2){\line(1,-2){0.5}} 
\put(1.1,0.2){\line(-1,0){1.0}} 

\put(0.0, 0.2){\makebox(0.0,0.0){$1$}}
\put(1.2, 0.2){\makebox(0.0,0.0){$2$}}
\put(0.6, 1.35){\makebox(0.0,0.0){$3$}}

\end{picture}~
\begin{picture}(2.3, 1.3)
\linethickness{0.2mm}

\put(0.6,1.2){\circle*{0.1}} 
\put(0.1,0.2){\circle*{0.1}} 
\put(1.1,0.2){\circle*{0.1}} 

\put(0.1,0.2){\line(1,2){0.5}} 
\put(0.6,1.2){\line(1,-2){0.5}} 

\put(0.0, 0.2){\makebox(0.0,0.0){$1$}}
\put(1.2, 0.2){\makebox(0.0,0.0){$2$}}
\put(0.6, 1.35){\makebox(0.0,0.0){$3$}}

\end{picture}
\end{minipage}

\end{center}
Let us fix two trees: $1\mapsto 2 \mapsto 3$ and $1\mapsto 3\mapsto 2$. Then formula (\ref{eqStationaryMeasureDetailedBalance}) yields $w^1 = (bc,be,ef)^T$ and $w^2 = (df,ab,af)$ and the proportionality factor is $e/a$.
\end{exam}

\end{document}